\documentclass[12pt,a4paper,leqno]{amsart}
\advance\textwidth by 2.5cm 
\advance\oddsidemargin by -1.6cm \advance\evensidemargin by -1.6cm

\newcommand{\lb}{\langle}
\newcommand{\rb}{\rangle}

\def\toup{\nearrow}

\newcommand\osc{\mathrm{osc}}

    \usepackage{amssymb}
\input colordvi

\newcommand\cyr{%
\renewcommand\rmdefault{wncyr}%
\renewcommand\sfdefault{wncyss}%
\renewcommand\encodingdefault{OT2}%
\normalfont \selectfont} \DeclareTextFontCommand{\textcyr}{\cyr}

\newcommand\del[1]{}
\newcommand\Reddel[1]{}

\def\eps{\varepsilon}

\theoremstyle{plain}
\newtheorem{theorem}{Theorem}[section]

\theoremstyle{remark}
\newtheorem{remark}[theorem]{Remark}

\theoremstyle{plain}
\newtheorem{corollary}{Corollary}[section]
\newtheorem{lemma}[theorem]{Lemma}
\newtheorem{proposition}[theorem]{Proposition}

\begin{document}
\baselineskip 17.5pt 
\numberwithin{equation}{section}

\title[OU process driven by a L{\'e}vy \del{white}noise]
{Time irregularity of generalized Ornstein--Uhlenbeck processes}

\author[Z Brze{\'z}niak]{Zdzis{\l}aw Brze{\'z}niak}
\address{Department of Mathematics\\
The University of York\\
Heslington, York YO10 5DD, UK} \email{zb500@york.ac.uk}

\author[B Goldys]{Ben Goldys}
\address{School of Mathematics\\
The University of New South Wales\\
Sydney 2052, Australia} \email{B.Goldys@unsw.edu.au}

\author[P Imkeller]{Peter Imkeller}
\address{Institut f\"ur Mathematik,
Humboldt-Universitaet zu Berlin, Unter den Linden 6
10099 Berlin, Germany}
\email{imkeller@mathematik.hu-berlin.de}

\author[S Peszat] {Szymon Peszat}
\address{Institute of Mathematics, Polish Academy of Sciences,
\'Sw. Tomasza 30/7, 31-027 Krak\'ow, Poland} \email{napeszat@cyf-kr.edu.pl}

\author[E Priola] {Enrico Priola}
\address{Dipartimento di Matematica, Universita di Torino,
via Carlo Alberto 10, 10-123 Torino, Italy} \email{enrico.priola@unito.it}

\author[J Zabczyk]{Jerzy Zabczyk}
\address{Institute of Mathematics, Polish Academy of Sciences, P-00-950 Warszawa, Poland}
\email{zabczyk@impan.pl}

\thanks{ \noindent \! \!   Supported by the Polish
Ministry of Science and Education project 1PO 3A 034 29 ``Stochastic
evolution equations with \! \! L\'evy noise'' and by EC FP6 Marie Curie ToK programme SPADE2}

\keywords{L{\'e}vy white noise, subordinator process, Langevin equation, space regularity, stochastic heat equation, non-c{\`a}dl{\`a}g trajectories,
stochastic Burgers equation}


\begin{abstract}
The paper is concerned  with the properties of solutions to linear
evolution equation perturbed by cylindrical L{\'e}vy processes. It turns out
that solutions, under rather weak requirements, do not have
c{\`a}dl{\`a}g modification. Some  natural open questions are
 also  stated.
\end{abstract}

\date{\today}

\maketitle

\section{Introduction}
\label{sec_intro}

In the study of spdes with L{\'e}vy noise a special role is played
by linear stochastic equations:
\begin{equation}
\label{eqn_langevin_000} \left\{\begin{array}{rcl}
dX(t)&=& A X(t)\,dt +dL(t),\; t\geq 0,\\
X(0)&=&0\in H,
\end{array}\right.
\end{equation}
on a Hilbert space $H$. In (\ref{eqn_langevin_000}) $A$ stands for
an infinitesimal generator of a $C_0$- semigroup $S(t)$ on $H$ and
$L$ is a L{\'e}vy process with values in   a Hilbert space $U$
often different and larger than $H$. The weak solution to
(\ref{eqn_langevin_000}) is of the form, see e.g.
\cite{Peszat_Zabczyk_2007},
\begin{equation}
\label{eqn_langevin_001} X(t) = \int_0^t S(t-s) dL(s),\,\, t\geq 0.
\end{equation}
The time regularity of the process $X$ is of prime interest in the
study of non-linear stochastic  PDEs, see e.g.
\cite{Priola+Z_2009}. If the L{\'e}vy process $L$ takes values in
$H$ than the solution $X$ has $H$-c{\`a}dl{\`a}g trajectories
because of the maximal inequalities for stochastic convolutions due
to Kotelenez \cite{Kotelenez_1987}, see also
\cite{Peszat_Zabczyk_2007}. However the process $X$  can take
 values in $H$   even if the space $U$ is larger than $H$ and $L$ does not
evolve in $H$. It turns out that if $L$ is the so called L\'evy
white noise the process $X$ does not have $H-$ c{\`a}dl{\`a}g
trajectories, see \cite{Brz+Z_2009} and \cite{Peszat_Zabczyk_2007},
although it may have a version taking values in a subspace of $H$ of
rather regular elements. The main reason for   this phenomenon
was the fact that the process $L$ had jumps not belonging to the
space $H$. It was therefore natural  to conjecture that if the jumps
of the process $L$ belong to $H$ than the c{\`a}dl{\`a}g
modification of $X$ should exist. The present paper shows that this
is not always  the case, and that in fact  the problem of characterizing equations
(\ref{eqn_langevin_000}) which solutions have c{\`a}dl{\`a}g
modification is still open. This is true even in the diagonal case
for which, in the Gaussian case, there are satisfactory answers, see
\cite{Iscoe_1990}, \cite{DaPrato+Zabczyk}. \vspace{3mm}

In the present paper we consider
 a class of processes $L$ which have expansion of the form
\begin{equation}
\label{eqn_langevin_002} L(t) = \sum_{n=1}^{\infty}\beta_n L^n (t) e_n ,\,\, t\geq 0.
\end{equation}
where $L^n$ are independent, identically distributed,
c{\`a}dl{\`a}g, real valued L\'evy processes with the jump intensity
$\nu$  not identically 0. Here $(e_n)$ is an orthonormal
basis in $H$ and $\beta_n$ is a sequence of positive numbers. It is
not difficult to see that the jumps of the process $L$ belong to $H$
but only under special assumptions the process $L$  evolves in $H$.
We show that, in general, the process $X$ does not have an $H$--
c{\`a}dl{\`a}g modification. The case of stochastic heat equation
will be considered with some detail.

\section{Main theorem}


The main result of the paper is the following theorem

\begin{theorem}\label{thm-non_cadlag1}
 Assume that the process $X$ in \eqref{eqn_langevin_001}
 is an  $H$-valued  process
   and that the elements of the basis $(e_n)$    belong
to the domain $D(A^*)$
 of the operator $A^*$ adjoint to $A$. If $\beta_n$
  do not converge to $0$,
  then, with probability 1, trajectories of $X$ have no point $t
   \in [0, +\infty)$
      in
  which there exists the left limit $X(t-) \in H$ or the right limit
 $X(t+) \in H$.
\end{theorem}

\begin{corollary}\label{thm-non_cadlag2} Assume that
 the hypotheses of Theorem  \ref{thm-non_cadlag2} hold.
  Then  the process $X$
  has no an $H$--c{\`a}dl{\`a}g modification.
\end{corollary}

\begin{remark}\label{rem-Priola+Z_2008p}Consider an important  case when the operator $A$ is self-adjoint with eigenvectors $e_n$ and
the corresponding eigenvalues $-\lambda_n
<0,\,\,n=1,2,\ldots $ tending to $-\infty$. Denote by $X^n$ the $\mathbb{R}$-valued Ornstein--Uhlenbeck  process defined by
\begin{equation}
\label{eqn_langevin_01} \left\{\begin{array}{rcl}
dX^n(t)&=&-\lambda_n X^n(t)\,dt +\beta_n dL^n(t),\; t\geq 0,\\
X^n(0)&=&0,
\end{array}\right.
\end{equation}
and identify $H$ with $l^2$. The regularity of the process $X(t) = (X^n(t))$ with  $L^n$ independent Wiener processes was considered in the paper
\cite{Iscoe_1990} were conditions, close to necessary and sufficient,  for continuity of trajectories, were given. If the processes $\big(L^n\big)$,
$n\in\mathbb{N}$,  are without gaussian part, i.e. of pure jump type,
 and $\nu$ is a symmetric measure, then the necessary and sufficient conditions for
the process $(X^n(t))$ to
 take  values in $l^2$ are given in a recent paper \cite{Priola+Z_2008p}.
\end{remark}

The proof of the theorem will be a consequence of two  lemmas.
 The first of them is a variant of the well-known Cauchy criterium
  for the existence of limit.

\begin{lemma}\label{lem}
 A function $f\colon [0, + \infty) \to E$, where $E$ is a  Banach space
(with norm denoted by $\Vert\cdot\Vert)$ admits left limit at some
 $t>0$
 (respec. right limit at some $s\ge 0$)
  if and only if
 for an arbitrary $\eps>0$ there exists $\delta >0$
  such that
\begin{equation}
\label{eqn_cadlag-var} \osc_f((t- \delta, t))<\eps \;\; (\! \mbox{
 respec.
  } \; \osc_f((s, s+ \delta))<\eps),
\end{equation}
where, for $\Gamma \subset [0,1]$,
$$\osc_f(\Gamma):= \sup_{s,t\in \Gamma}\Vert
f(t)-f(s)\Vert.$$
\end{lemma}

\begin{lemma}\label{lem-3}
Assume that for some $0<r_1<\infty$,
 $\nu((-\infty,-r_1] \cup [r_1,\infty))>0$.
 Let $\tau_{n}$ denote the {\it first}  jump
  of the process $L_n$ of magnitude
at least $r_1$, in particular
$$\vert\Delta L^n(\tau_{n})\vert \geq r_1.$$
Then, with probability $1$,  the set
$$\{\tau_{n}\colon n \in\mathbb{N}^\ast\}$$
is dense in the interval $(0, + \infty)$.
\end{lemma}

\begin{proof} Recall that   $\tau_{n}$, $n\in\mathbb{N}^\ast$,
  are independent and
 exponentially distributed with  parameter $\lambda =
  \nu ( \{ t \in \mathbb{R} \colon  |t| \geq r_1 \}  )$
 (independent of $n$). Let
 $\alpha,\beta \in \mathbb{Q}$ be such that $0<\alpha<\beta $ and
let
$$
 A_n:=\{ \tau_{n} \in (\alpha,\beta)\}.
$$
Then, for any $n \ge 1$,  $\mathbb{P}(A_n) = \mathbb{P}(A_1)\in
(0,1)$, $A_n$ are independent events  and $\sum_{n=1}^\infty
\mathbb{P}(A_n) = \infty$.  Consequently, by the second
Borel--Cantelli Lemma, with probability $1$, there exists $n
\in\mathbb{N}^\ast$, such that $\tau_n \in (\alpha,\beta)$. Since
the family
$$
\{(\alpha,\beta)\colon  \alpha,\beta \in \mathbb{Q}, \;
0<\alpha<\beta\}
$$
is countable, the result follows.
\end{proof}

\begin{proof}[Proof of Theorem \ref{thm-non_cadlag1}]

Since  $X$ is  a weak solution, see e.g. \cite{Peszat_Zabczyk_2007},
for each $n$,
 \begin{equation} \label{gt}
d \langle X(t), e_n \rangle _H = \langle X(t), A^*e_n \rangle _H dt + \beta_n dL^n (t).
 \end{equation}
 Denote the processes $\langle X(t), e_n \rangle _H $ by $X^n(t)$.

Passing to subsequences we can assume  that for some $r_2>0$ and for
all $n$, $\beta_n\geq r_2$.  Let $\tau_{n}$ denote the moment of the
 first jump of the process $L_n$ of the absolute size greater
than or equal  $r_1$. These numbers form, with
probability $1$ (say, for any $\omega \in \Omega_0 $) a dense subset
of the interval $(0,+ \infty)$
 (see Lemma \ref{lem-3}), and, at
 each moment $\tau_{n}$, the process $\beta_n L_n$ has a jump of the
 absolute size at least $r_1r_2$.

 Arguing by contradiction, let us assume that
 there exists $\omega \in \Omega_0 $ such that
  at   some time  $ t_{\omega}$
  the left limit or the right limit of $X(\cdot , \omega)$
   exists.
  Let us  assume that $t_{\omega} \ge 0 $ and that
  there exists the right limit at $t_{\omega}$.
 By Lemma \ref{lem} this means that, for any $\epsilon >0$,
 there exists  $\delta >0$ such that
 \begin{equation} \label{fr}
 osc_{ X (\cdot, \omega)}((t_{\omega} ,t_{\omega}
   + \delta ))<\eps.
\end{equation}
 Let $\eps>0$ be any number smaller than $r_1r_2$ and let
 $\delta  >0$ be such that \eqref{fr} holds.
   There exists a natural number $n_0$ (depending also on $\omega$)
     such that
 $$\tau_{n_0}(\omega)\in (t_{\omega} ,t_{\omega}
   + \delta ).
    $$
 Note that  $\Vert X(t, \omega)- X(s, \omega) \Vert \geq |X^n(t, \omega)-
 X^n(s, \omega)|$, for any $n \ge 1$, $t, s \ge 0$. Using  also equation
  \eqref{gt}, we infer
$$\liminf_{s\toup \tau_{n_0}(\omega)}
 \Vert X(\tau_{n_0}(\omega), \omega)-X(s, \omega)\Vert
 \ge \liminf_{s\toup \tau_{n_0}(\omega)}
 | X^{n_0}(\tau_{n_0}(\omega), \omega)-X^{n_0}(s, \omega) |
 $$$$= \liminf_{s\toup \tau_{n_0}(\omega)}
 |  \beta_{n_0} L_{n_0}(\tau_{n_0}(\omega), \omega)- \beta_{n_0}
  L_{n_0}(s, \omega) |
 \geq r_1r_2 >\eps,
$$
 which contradicts the statement \eqref{fr}.
\end{proof}

The result raises some natural questions.\vspace{3mm}

\emph{ Question 1.} Does the assumption that the sequence $(\beta_n)$
tends to zero imply existence of a c{\`a}dl{\`a}g modification of
$X$?\vspace{3mm}

\emph{ Question 2.} Does the assumption that $e_n \in D(A^*)$ is essential for the validity of the Theorem \ref{thm-non_cadlag1}?\vspace{3mm}

\emph{ Question 3.} Is the requirement that the process $L$ evolves in $H$ also necessary for the existence of $H$- c{\`a}dl{\`a}g modification of $X$?
\vspace{3mm}

\section{Heat equation with $\alpha$--stable noise}

In the present section we assume that  $A=\Delta$ is the Laplace operator with the Dirichlet boundary conditions on $\mathcal{O}=(0,\pi)$. Let  $H=
L^2(\mathcal{O})$ and
\begin{equation}
\begin{array}{rclrcl}
D(A)&=&H^2(\mathcal{O})\cap H_0^1(\mathcal{O}),\;\; Au&=&\Delta u,\; u\in D(A).
\end{array}
\end{equation}

It is well known that $A$ is a self-adjoint negative operator on $H$ and that $A^{-1}$ is compact. Hence $A$ is  of diagonal type, with respect to
eigenfunctions $\big(e_j\big)_{j=1}^\infty$, where
$$
e_j (\xi) = \sqrt\frac{2}{\pi} \sin( j \xi), \;\; \xi \in \mathcal{O},\; j\in\mathbb{N}^\ast.$$
 The
corresponding eigenvalues of the operator $-A$ are
$$
\lambda_j = j^2, \; j\in\mathbb{N}^\ast.
$$
Setting  $\beta_j=1$,  $j\in\mathbb{N}$, and assuming that  $L^n$ are independent, identically distributed, c{\`a}dl{\`a}g, real valued $\alpha -$ stable
L\'evy processes, $\alpha \in (0,2]$,  we will work with  a ``white'' $\alpha -$ stable  process:
$$L(t)=\sum_{j=1}L^j(t)e_j, \; t\geq 0,$$ and with the solution $X$ of
\begin{equation}\label{e1}
dX(t)= \Delta X(t)+dL(t),\; t\geq 0,\;\; X(0)=0.
\end{equation}

\noindent For $\delta \geq 0$, define $H_\delta=D(A^{\delta/2})$ with the naturally defined scalar product. Then in particular $H_0=H$, $H_1=H_0^1(\mathcal{O})$ and $H_2=D(A)$
 and moreover,
$$H_\delta=\Big\{x\in H\colon  \sum_{j=1}^\infty \lambda_j^\delta |x_j|^2<\infty \Big\}
$$
where $x_j:=\lb x,e_j\rb$, $j\in \mathbb{N}^\ast$. For  $\delta < 0$, by $H_\delta$ we denote  the extrapolation space which can be  defined as
$D(A^{-\delta/2})$, or more precisely as the completion of the space $H$ with respect to the norm
$ \vert x\vert_\delta:=\vert A^{-\delta/2}x\vert, \; x\in H$.
The Hilbert space $H_\delta$ can then be isometrically identified with the weighted space $l^2_{\delta}$,
$$l^2_{\delta} = \Big\{x=(x_j)\colon  \sum_{j=1}^\infty \lambda_j^\delta |x_j|^2<\infty \Big\},$$
equipped with the norm $ \vert x\vert_\delta:= (\sum_{j=1}^\infty \lambda_j^\delta |x_j|^2)^{1/2}$.

\begin{proposition}\label{P1} Assume that $X$ solves equation (\ref{e1}) and  $\alpha \in (0,2)$. Then:

i) The process $L$ is \,\, $H_{\delta}$-- valued, and thus $H_{\delta}$--c{\`a}dl{\`a}g, \,\,if and only if \,\, $\delta < - 1/{\alpha}$.

ii)The process  $X$ is \,\,$ H_{\delta}$--valued \,\,if and only if \,\, \,\,$\delta <  1/{\alpha}$.

iii) If $ \delta < - 1/{\alpha}$ then the process $X$ is $H_{\delta}$--c{\`a}dl{\`a}g.

iv) If $ \delta \geq 0$ then the process $X$ has no $H_{\delta}$ --c{\`a}dl{\`a}g modification.
\end{proposition}
\begin{proof}  $i)$ It follows  from \cite[Proposition 3.3]{Priola+Z_2009} that the process $L$ takes values in the space $H_\delta$ if and only if
$$
\sum \vert \lambda_j^{\delta/2}\vert^\alpha<\infty.
$$
Since $\lambda_j=j^2$, $\vert \lambda_j^{\delta/2}\vert^\alpha=j^{\delta\alpha}$, we infer that the
process $L$ takes values in the space $H_\delta$ if and only if  $\delta<-1/\alpha$.

$ii)$ The argument from the proof of $i)$ applies.

$iii)$ By the maximal inequalities for stochastic convolution, \cite{Kotelenez_1987},  we infer that if the procees $L$ is $H_{\delta}$--valued then the process $X$ is $H_\delta$-c{\`a}dl{\`a}g.

$iv)$ This is a direct consequence of our Theorem \ref{thm-non_cadlag1}.
\end{proof}
It is of interest to compare the stable case $\alpha \in (0,2)$ with the Gaussian case $\alpha = 2$.
\begin{proposition}\label{P2} Assume that $X$ solves equation (\ref{e1}) and  $\alpha = 2$. Then

i) The process $L $ is $H_{\delta}$-- valued, and thus $H_{\delta}$--continuous, \,\,if and only if \,\, $\delta < - 1/2$.

ii)The process  $X$ is $H_{\delta}$-- valued   \,\,if and only if \,\, $\delta < 1/2$.

iii)The process is  $X$  is $H_{\delta}$-- continuous \,\,if and only if \,\,$\delta < 1/2$.
\end{proposition}
\begin{proof} This is a well known result. Parts $i)$ and $ii)$ can be proved in the same way as in the previous theorem. To prove $iii)$ it is enough to apply a sufficient condition for continuity from \cite{DaPrato+Zabczyk}, namely that $\exists {\beta >0},\exists T > 0 $ such that
$$
\int_0^T t^{-\beta} \|e^{t\Delta}\|^2_{L_{HS}(H_0 , H_{\delta})}dt <+\infty,
$$
where $\|\cdot\|_{L_{HS}(H_0 , H_{\delta})}$ denotes the Hilbert--Schmidt norm of an operator from $H_0$ into $H_{\delta}$. One can also use \cite{Iscoe_1990}.
\end{proof}

We see that the regularity result for the Gaussian Ornstein--Uhlenbeck process does not have a precise analog for the  $\alpha$--stable process. We
have the following natural open question where  our Theorem \ref{thm-non_cadlag1} is here not applicable.

\emph{ Question 4.} Is the process $X$ from Proposition \ref{P1}, $H_\delta$-c{\`a}dl{\`a}g for $\delta \in [-1/\alpha,0)$?

\end{document}